\documentclass[12pt]{amsart}

\usepackage{amsfonts}
\usepackage{amsmath}
\usepackage{amssymb}
\usepackage{amstext}
\usepackage{amsthm}
\usepackage{mathtools}
\usepackage{amsopn}
\usepackage{mathrsfs}
\usepackage{geometry}
\usepackage{url}
\usepackage{amsrefs}
\usepackage{hyperref}
\usepackage[english]{babel}
\usepackage[autostyle, english = american]{csquotes}
\MakeOuterQuote{"}

\hypersetup{
	colorlinks=true,
	linkcolor=red,
	filecolor=magenta,      
	urlcolor=cyan,
}

\theoremstyle{definition}
\theoremstyle{remark}
\newtheorem{theorem}{Theorem}[section]

\newtheorem{lemma}[theorem]{Lemma}
\newtheorem{example}[theorem]{Example}
\newtheorem{definition}[theorem]{Definition}
\newtheorem{corollary}[theorem]{Corollary}
\newtheorem{remark}[theorem]{Remark}

\numberwithin{equation}{section}

\DeclareMathOperator{\lcm}{lcm}

\DeclareMathOperator{\lm}{lm}
\DeclareMathOperator{\LM}{LM}
\DeclareMathOperator{\lt}{lt}
\DeclareMathOperator{\lti}{\boldsymbol{lt}}
\DeclareMathOperator{\lmi}{\boldsymbol{lm}}
\DeclareMathOperator{\lci}{\boldsymbol{lc}}
\DeclareMathOperator{\lc}{lc}
\DeclareMathOperator{\LT}{LT}
\DeclareMathOperator{\syz}{Syz}
\DeclareMathOperator{\ann}{ann}

\begin{document}
	
\title{Gr\"{o}bner bases and syzygy theorem for direct product of principal ideal rings}

\author{Babak Jabarnejad}
\address{Department of Mathematical sciences, University of Arkansas, Fayetteville, Arkansas, 72701, USA}
\email{babak.jab@gmail.com}


\date{}
	
\dedicatory{}
	
\begin{abstract}
In this paper we give versions of Hilbert's syzygy theorem for finitely generated modules over polynomial rings over direct product of principal ideal rings.
\end{abstract}
	
\maketitle
	
\section{Introduction}

The concept of Gr\"{o}bner bases for polynomial rings over a field was presented by Buchberger~\cite{buchberger1965algorithmus}. He also gave generalizations of this concept over some rings (e.g.~\cite{buchberger1984critical}). Some generalizations of this concept can be found in~\cite{adams1994introduction}, \cite{yengui2006dynamical}, \cite{kacem2010dynamical}, \cite{hofmann2019gr}, \cite{eder2019efficient}.

In the excellent paper~\cite{gamanda2019syzygy}, Gamanda et.\ al.\ generalize the concept of Gr\"{o}bner bases for B\'ezout rings with divisibility test. They also give versions of Hilbert's syzygy theorem for B\'ezout domains and B\'ezout rings of dimension zero with divisibility test. 

In the present paper, we give a generalization to syzygies on polynomials whose leading coefficients are monomials in a fixed permutable weak regular sequence. We show that length of a free resolution of monomial ideal in permutable weak regular sequence $s_1,\dots,s_k$ is bounded by $k$. Later on as an application we give different versions of Hilbert's syzygy theorem for polynomial rings over principal ideal domains and direct product of principal ideal rings. From the proofs we conclude that if $M\cong F/U$ ($F$ is a free module) and leading coefficients of a Gr\"{o}bner basis of $U$ are nonzero divisors, then the module $M$ admits a finite free resolution. A result that is not proved in~\cite{gamanda2019syzygy}. Note that By \cite[Lemma 10, Corollary 11]{hungerford1968structure}, every principal ideal ring is a finite direct sum of quotients of PIDs. So to prove facts for principal ideal rings it is enough to prove facts for finite direct sum of quotients of PIDs. Also, we know that by Auslander-Buchsbaum-Serre Theorem every finitely generated module over a local ring has a finite free resolution iff the ring is regular. But this theorem does not hold in general for a non-local regular ring. For example if we consider regular ring $\mathbb{Z}\times\mathbb{Z}$ then the ideal $\mathbb{Z}\times\{0\}$ doesn't have a finite free resolution.

The paper~\cite{gamanda2019syzygy} lies in the framework of constructive mathematics, but our paper lies in the framework of classical mathematics. Therefore, we don't assume that rings are equipped with divisibility test. Additionally, we point out that in the present paper we don't need to assume that the direct product of principal ideal rings have dimension zero. As our method of proofs are totally different than methods of proofs in~\cite{gamanda2019syzygy}. They localize B\'ezout rings whereas we directly use the concept of permuatble weak regular sequence.  

In the last section of the paper we generalize the Gr\"{o}bner bases concept for polynomial rings over solvable principal ideal rings. We end the paper with some examples, including an example, where we compute a free resolution on a polynomial ring over a principal ideal ring of dimension 1.  

\section{Gr\"{o}bner bases on finitely generated modules}

In the entirety of this paper $R$ is a commutative ring with unit. We fix $S=R[x_1,...,x_n]$ and finitely generated free $S$-module $F=\oplus_{i=1}^{r}S$ with standard basis $\{e_1,\dots,e_r\}$. Every monomial of $F$ has the form $ue_i$, where $u$ is a monomial of $S$. Assume we have a monomial order $>$ on $F$. If $f\in F$, then we denote the leading term of $f$ by $\LT(f)$. If $\LT(f)=aue_i$, where $a\in R$ and $u$ is a monomial in $S$, then we denote $\lc(f)=a$, $\LM(f)=ue_i$, $\lm(f)=u$ and $\lt(f)=au$. For every $f\neq 0$, we decide $\LT(f)>\LT(0)=0$. For $aue_i$ and $bve_j$, where $a,b\in R$ and $u,v$ are monomials in $S$, we say that $aue_i$ divides $bve_j$ and we write $aue_i|bve_j$ if $au$ divides $bv$ and $i=j$.

\begin{example}
Let $S=R[x_1,x_2,x_3]$ and we have lexicographic order on $S$. Suppose that $F=S\oplus S\oplus S$. Assume $F$ has lexicographic order given priority to the position. If $f=rx_1^3x_2e_1-x_1^4x_2x_3e_3$, then $\LT(f)=rx_1^3x_2e_1$, $\lt(f)=rx_1^3x_2$, $\LM(f)=x_1^3x_2e_1$, $\lc(f)=r$ and $\lm(f)=x_1^3x_2$.
\end{example}

\begin{remark}
In this paper the monomial order on $F$ does not depend on an order on $S$. Let $f\in F$ and $g\in S$. Suppose $\{b_lf_le_j\}$ is the set of all terms of $f$ ($b_l\in R$ and $f_le_j$ is a monomial of $F$) and $\{a_ig_i\}$ is the set of all terms of $g$ ($a_i\in R$ and $g_i$ is a monomial of $S$). Then there is a single $g_{i_0}$ and a single $f_{l_0}e_{j_0}$ such that $g_{i_0}f_{l_0}e_{j_0}$ is maximal among all $g_if_le_j$, because if there are $g_{i_1},g_{i_2}$ and $f_{l_1}e_{j_1},f_{l_2}e_{j_1}$ such that $g_{i_1}f_{l_1}e_{j_1}=g_{i_2}f_{l_2}e_{j_1}$, assume that $f_{l_1}e_{j_1}>f_{l_2}e_{j_1}$, then $g_{i_2}f_{l_1}e_{j_1}>g_{i_2}f_{l_2}e_{j_1}$, so such a $g_{i_1}f_{l_1}e_{j_1}$ cannot be maximum. We denote $\lci(g)=a_{i_0}$, $\lmi(g)=g_{i_0}$ and $\lti(g)=a_{i_0}g_{i_0}$. Actually the monomial order on $F$ induces a monomial order on $S$ and with this monomial order $\lmi(g)$ is the leading monomial of $g$ and also we see that $\LM(f)=f_{l_0}e_{j_0}$. If $R$ is an integral domain, then $\LM(gf)=\lmi(g)\LM(f)$, otherwise we have $\lmi(g)\LM(f)\ge \LM(gf)$.
\end{remark}

\begin{definition}
Let $G=\{f_1,...,f_m\}\subseteq F$. We say $f\in F$ reduces to zero modulo $G$ and denote this by $f \rightarrow_{G} 0$ if there are $p_i\in S$, for $1\le i \le m \ $, such that
$$
f=p_1f_1+...+p_mf_m,\ \text{and} \ \LM(f) \ge \lmi(p_i)\LM(f_i)  
$$
\end{definition}

\begin{definition}
Let $U$ be a submodule of $F$ and $\{f_1,...,f_m\}\subseteq U$. We say $\{f_1,...,f_m\}$ is a Gr\"{o}bner basis for $U$ if $\langle \LT(f_1),...,\LT(f_m)\rangle=\LT(U)$, where $\LT(U)$ is the $S$-module generated by the leading terms of elements of $U$.
\end{definition}

If $R$ is a Noetherian ring and $U$ is a submodule of $F$, then a Gr\"{o}bner basis of $U$ always exists.

\begin{lemma}\label{groebner-generator}
If $U$ is a submodule of $F$ and $G=\{f_1,...,f_m\}$ is a Gr\"{o}bner basis for $U$, then for every $f\in U$, we have $f\to_G0$.
\end{lemma}
\begin{proof}
If $f\in U$, then we have an expression $\LT(f)=\sum r_ig_i\LT(f_i)$, where $r_i\in R$ and $g_i$ are monomials in $S$. If we take the element $h_1=\sum r_ig_if_i$, then $h_1\in\langle f_1,\dots,f_m\rangle\subseteq U$ and $\LT(f)=\LT(h_1)\ge\lmi(r_ig_i)\LM(f_i)$, so that $\LM(f)>\LM(f-h_1)$. We repeat this procedure and finally we have $\LT(f-h_1-\dots-h_{k-1})=0$ or $\LT(f-h_1-\dots-h_{k-1})=a\in R$. If $\LT(f-h_1-\dots-h_{k-1})=0$, then $f-h_1-\dots-h_{k-1}=0$, and so $f=h_1+\dots+h_{k-1}$. If $\LT(f-h_1-\dots h_{k-1})=a$, then $f-h_1-\dots-h_{k-1}=a$. But $a$ can be generated by $f_i$. This completes the proof.
\end{proof}

\begin{remark}
From this lemma we conclude that if $U$ is a submodule of $F$ and $G=\{f_1,...,f_m\}$ is a Gr\"{o}bner basis for $U$, then the elements of $G$ generate $U$. 
\end{remark}

\begin{definition}
Let $U$ be a submodule of $F$ generated by $f_1,\dots,f_m$. Suppose $G$ is a free module over $S$ with the standard basis $g_1,\dots,g_m$ and let $\alpha:G\to U$ be a homomorphism such that $\alpha(g_i)=f_i$. The Schreyer monomial order on $G$ induced by $f_1,\dots,f_m$ and $<$ (a given monomial order in $F$) and denoted by $<_{f_1,\dots,f_m}$ is defined as the following: Let $ug_i$ and $vg_j$ be monomials in $G$. Then we set
$$
ug_i<_{f_1,\dots,f_m}vg_j\iff \LM(uf_i)<\LM(vf_j), \text{or} \LM(uf_i)=\LM(vf_j)\ \text{and}\ j<i.
$$ 
\end{definition}
This order is a monomial order, see~\cite{ene2011gr}.
	
\section{Syzygies of $\underline{s}$-monomial type elements}

As we know the sequence $\underline{s}=s_1,\dots,s_k$ is a weak regular sequence if $s_i$ is an $R/\langle s_1,\dots,s_{i-1}\rangle$-regular element for $i=1,\dots,k$. We fix a permutable weak regular sequence $\underline{s}=s_1,\dots,s_k$ in $R$, where none of them is a unit. By an $\underline{s}$-monomial we mean a monomial in this fixed weak regular sequence and by an $\underline{s}$-term we mean an element of the form $ua$, where $u$ is a unit and $a$ is an $\underline{s}$-monomial. 
	
\begin{remark}
Any $\underline{s}$-monomial has a unique representation. Also, we consider 1 as an $\underline{s}$-monomial, even though 1 is also a monomial in $S$.
\end{remark}
	
\begin{remark}
If $a,b\in R$ and $a,b$ are $\underline{s}$-terms, then $\lcm(a,b)$ and $\gcd(a,b)$ exist and they are unique up to unit (by $\lcm$ (resp. $\gcd$) we mean the least common multiple (resp. the greatest common divisor)). If $a=u_1s_1^{\alpha_1}\dots s_k^{\alpha_k},\ b=u_2s_1^{\beta_1}\dots s_k^{\beta_k}$, $\alpha_i,\beta_i\in\mathbb{Z}_{\ge 0}$ are $\underline{s}$-terms, then we canonically choose $\gcd(a,b)=s_1^{\gamma_1}\dots s_k^{\gamma_k}$ and $\lcm(a,b)=s_1^{\eta_1}\dots s_k^{\eta_k}$, where $\gamma_i=\min\{\alpha_i,\beta_i\}$ and $\eta_i=\max\{\alpha_i,\beta_i\}$.
\end{remark}

\begin{corollary}\cite[Corollary 2.6]{jabarnejad2017rees}\label{syzygy}
Let $a_1,\dots,a_n\in R$ be $\underline{s}$-monomials. Then the 
$R$-module $\{(c_1,\dots,c_n);\ c_1a_1+\dots+c_na_n=0\}\subseteq R^n$ is generated by $\{\frac{\lcm(a_i,a_j)}{a_i}e_i-\frac{\lcm(a_i,a_j)}{a_j}e_j;\ 1\le i < j\le n\}$.
\end{corollary}

\begin{definition}
If $f \in F$, and $\lc(f)$ is an $\underline{s}$-term, then we say $f$ is $\underline{s}$-monomial type.
\end{definition}

Let $f,g$ be $\underline{s}$-monomial type elements. If the leading terms of $f$ and $g$ don't involve the same basis element, then we don't define the $S$-element, otherwise we have
$$
S(f,g)=\frac{\lcm(\lm(f),\lm(g))}{\lt(f)}f-\frac{\lcm(\lm(f),\lm(g))}{\lt(g)}g.
$$
It is clear that the $S$-element of such $f$ and $g$ is in a finitely generated free module over $K[x_1,...,x_n]$, where $K$ is the total ring of fractions of $R$. Now, we define the 
$S^{'}$-element for such $\underline{s}$-monomial type elements in $F$.

$$
S^{'}(f,g)=\lcm(\lc(f),\lc(g))S(f,g).
$$ 
It is clear that $S^{'}(f,g)\in F$.

The following lemma is the key result that will be used to prove Buchberger's criterion for our case. But before we start we observe a notation: If $\delta=(\delta_1,\dots,\delta_n)\in \mathbb{Z}_{\ge 0}^{n}$, then by $x^{\delta}$ we mean $x_1^{\delta_1}\dots x_n^{\delta_n}$.  

\begin{lemma}\cite[Lemma 2.8]{jabarnejad2017rees}\label{pre-main-lemma1}
Let $f_1,\dots,f_m\in F$ be $\underline{s}$-monomial type. Suppose we have a sum $\sum_{i=1}^{m}c_if_i$, where $c_i\in R$ and for all $i$, $\LM(f_i)=x^{\delta}e_l, \delta\in \mathbb{Z}_{\ge 0}^{n}$. If $x^{\delta}e_l>\LT(\sum_{i=1}^{m}c_if_i)$, then $\sum_{i=1}^{m}c_if_i$ is an $R$-linear combination of the $S^{'}$-elements $S^{'}(f_i,f_j)$, for $1\le i<j \le m$. We also have $x^{\delta}e_l>\LM\left(S^{'}(f_i,f_j)\right)$.
\end{lemma}

The proof of the following theorem is similar to \cite[Theorem 6]{cox2007ideals}, but since we will use the proof later on we provide the proof.

\begin{theorem}[Buchberger's criterion]\label{Grobner-basis}
Let $G=\{f_1,...,f_m\}$ be a family of $\underline{s}$-monomial type elements in $F$ and $U=\langle f_1,...,f_m\rangle$. Then $G$ is a Gr\"{o}bner basis for $U$ iff for every $1\le i<j\le m$, $S^{'}(f_i,f_j)\rightarrow_G 0$.
\end{theorem}
\begin{proof}
Using Lemma~\ref{pre-main-lemma1}, the proof will proceed. Let $f\in U$, we consider all possible ways that $f=\sum_{i=1}^{m}g_if_i$, where $g_i\in S$, and if $x^{u(i)}e_{u(i)}=\lmi(g_i)\LM(f_i)$ and $x^\delta e_l=\max (x^{u(1)}e_{u(1)},\dots,x^{u(m)}e_{u(m)})$, so $x^\delta e_l$ is minimal. Clearly $x^{\delta}e_l\ge \LT(f)$. If $\LM(f)=x^{\delta}e_l$, then $\LT(f)\in \langle \LT(f_i)\rangle_{i=1}^{m}$, if not, thus $x^{\delta}e_l>\LT(f)$. Then we can write
\begin{gather*}
f=\sum_{x^{u(i)}e_{u(i)}=x^\delta e_l}g_if_i + \sum_{x^{u(i)}e_{u(i)}<x^\delta e_l}g_if_i\\
= \sum_{x^{u(i)}e_{u(i)}=x^\delta e_l}\lti(g_i)f_i+\sum_{x^{u(i)}e_{u(i)}=x^\delta e_l}(g_i-\lti(g_i))f_i+\sum_{x^{u(i)}e_{u(i)}<x^\delta e_l}g_if_i.
\end{gather*}
The $x^\delta e_l$ is greater than all monomials appearing in the fourth and fifth sums. Then the assumption $x^\delta e_l>\LT(f)$, means that the $x^\delta e_l$ is also greater than leading term of the third sum. 
	
Let $\lti(g_i)=c_ix^{\alpha(i)}$, then $\sum_{x^{u(i)}e_{u(i)}=x^\delta e_l}\lti(g_i)f_i=\sum_{x^{u(i)}e_{u(i)}=x^\delta e_l}c_ix^{\alpha(i)}f_i$, we assume $x^{\alpha(i)}f_i=h_i$, by Lemma~\ref{pre-main-lemma1},
 with a suitable order, $\sum_{x^{u(i)}e_{u(i)}=x^\delta e_l}c_ix^{\alpha(i)}f_i$ is a linear combination of $S^{'}(x^{\alpha(j)}f_j,x^{\alpha(k)}f_k)$, but 
$$
S(x^{\alpha(j)}f_j,x^{\alpha(k)}f_k)=x^{\delta - \gamma_{jk}}S(f_j,f_k)
$$
where $x^{\gamma_{jk}}=\lcm(\lm(f_j),\lm(f_k))$. Then
\begin{gather*}
S^{'}(x^{\alpha(j)}f_j,x^{\alpha(k)}f_k)=x^{\delta - \gamma_{jk}}S^{'}(f_j,f_k)\\
\Rightarrow \sum_{x^{u(i)}e_{u(i)}=x^\delta e_l}c_ix^{\alpha(i)}f_i=\sum_{x^{u(i)}e_{u(i)}=x^\delta e_l}c_{jk}x^{\delta -\gamma{jk}}S^{'}(f_j,f_k) \quad c_{jk}\in R.
\end{gather*}
On the other hand
\[
S^{'}(f_j,f_k)=\sum_{t=1}^{m}l_{tjk}f_t , l_{tjk}\in S,\ \LM(S^{'}(f_j,f_k)) \ge \lmi(l_{tjk})\LM(f_t)  
\]
thus
\begin{gather*}
x^{\delta - \gamma{jk}}S^{'}(f_j,f_k)=\sum_{t=1}^{m}a_{tjk}f_t, \ a_{tjk}=x^{\delta - \gamma_{tjk}}l_{tjk}\in S\\
x^\delta e_l>\LM(x^{\delta - \gamma_{jk}}S^{'}(f_j,f_k))\ge \lmi(a_{tjk})\LM(f_t).
\end{gather*}
We get the equation
\begin{gather*}
\sum_{x^{u(i)}e_{u(i)}=x^\delta e_l}c_{jk}x^{\delta -\gamma{jk}}S^{'}(f_j,f_k)=\sum_{j,k}c_{jk}\left( \sum a_{tjk}f_t\right)\\ 
=\sum \overline{h}_t f_t, \quad \overline{h}_t\in S, \quad x^\delta e_l>\lmi(\overline{h}_t)\LM(f_t). 
\end{gather*}
Therefore we have a contradiction.
\end{proof}

In the rest of this section our results are similar to some of the results in sections 4.4.1 and 4.4.3 in~\cite{ene2011gr} with slightly different proofs.

Let $G=\{f_1,...,f_m\}$ be a family of $\underline{s}$-monomial type elements in $F$ and they form a Gr\"{o}bner basis for submodule $U$. Then we have 
$$
S^{'}(f_i,f_j)= u_{ij}f_i-u_{ji}f_i=q_{ij,1}f_1+\dots+q_{ij,m}f_m,
$$
where $\LM(S^{'}(f_i,f_j))\ge\lmi(q_{ij,l})\LM(f_l)$, $i<j$, $u_{ij}=\frac{\lcm(b_i,b_j)}{c_ib_i}\frac{\lcm(u_i,u_j)}{u_i}$ ($c_ib_i=\lc(f_i)$, $c_i$ is a unit, $b_i$ is an $\underline{s}$-monomial and $u_i=\lm(f_i)$) and leading terms of such $f_i$ and $f_j$ involve the same basis element.

Let $g_1,\dots,g_m$ be a basis for a free module $L$ over $S$. We define
$$
r_{ij}=u_{ij}g_i-u_{ji}g_j-q_{ij,1}g_1-\dots-q_{ij,m}g_m.
$$

\begin{theorem}\label{syzygy-polynomial}
With mentioned notation and condition above $r_{ij}$ ($1\le i<j\le m$), generate the syzygies of $\{f_1,\dots,f_m\}$. 
\end{theorem}
\begin{proof}
Let $V$ be a submodule of $L$ generated by all $r_{ij}$ and let $G=\syz(f_1,\dots,f_m)$. We need to prove that $G\subseteq V$. Let $r=\sum_{j=1}^{m}h_jg_j\in G$.  

Let $w_r=\max\{\lmi(h_j)\LM(f_j); j=1,\dots,m\}$. Without loss of generality, we assume that $w_r=\lmi(h_j)\LM(f_j)$ for $j=1,\dots,t$ and $\lmi(h_j)\LM(f_j)<w_r$ for $j=t+1,\dots,m$. If $\lmi(h_j)=v_j$, and $\lci(h_j)=a_j$, then $w_r=u_jv_je_i$ for $j=1,\dots,t$, ($\LM(f_j)=u_je_i$) and $\sum_{j=1}^{t}a_jc_jb_j=0$. Hence by Corollary~\ref{syzygy}, we have
\begin{gather*}
(a_1c_1,\dots,a_tc_t)=\sum_{1\le i<j\le t}f_{ij}\frac{\lcm(b_i,b_j)}{b_i}k_i-f_{ij}\frac{\lcm(b_i,b_j)}{b_j}k_j,
\end{gather*}
where $\{k_i|1\le i\le t\}$ is the standard basis for $R^t$.

For each fixed $1\le i\le t-1$, we have $u_iv_i=u_jv_j$ for $j=i+1,\dots,t$ so that there exist monomials $w_{ij}$ such that $v_j=w_{ij}\frac{\lcm(u_i,u_j)}{u_j}$ for $j=i+1,\dots,t$. 

Assume for the moment that we already know that $w_{r^{'}}<w_r$ for the relation
$$
r^{'}=r-\sum_{i=1}^{t-1}\sum_{j=i+1}^{t}f_{ij}w_{ij}r_{ij}.
$$
By induction we may then assume that $r^{'}\in V$, which then implies that $r\in V$, since $\sum_{i=1}^{t-1}\sum_{j=i+1}^{t}f_{ij}w_{ij}r_{ij}\in V$. Hence it is enough to prove that $w_{r^{'}}<w_r$. Let $\rho_{ij}=\sum_{l=1}^{m}q_{ij,l}g_l$, where $\sum_{l=1}^{m}q_{ij,l}f_l$ is the 
standard expression of $S^{'}(f_i,f_j)$. Then we have
\begin{gather*}
r^{'}=\sum_{j=1}^{t}h_jg_j+\sum_{j=t+1}^{m}h_jg_j\\
-\sum_{i=1}^{t-1}\sum_{j=i+1}^{t}f_{ij}w_{ij}\left(\frac{\lcm(b_i,b_j)}{c_ib_i}\frac{\lcm(u_i,u_j)}{u_i}g_i-\frac{\lcm(b_i,b_j)}{c_jb_j}\frac{\lcm(u_i,u_j)}{u_j}g_j\right)+\sum_{i=1}^{t-1}\sum_{j=i+1}^{t}f_{ij}w_{ij}\rho_{ij}.
\end{gather*}
On the other hand we have
$$
w_{ij}\frac{\lcm(u_i,u_j)}{u_i}=\frac{w_{ij}u_j}{u_i}\frac{\lcm(u_i,u_j)}{u_j}=\frac{u_jv_j}{u_i}=\frac{u_iv_i}{u_i}=v_i,
$$
then we have
\begin{gather*}
r^{'}=\sum_{j=1}^{t}h_jg_j+\sum_{j=t+1}^{m}h_jg_j\\
-\sum_{i=1}^{t-1}\sum_{j=i+1}^{t}v_if_{ij}\frac{\lcm(b_i,b_j)}{c_ib_i}g_i-v_jf_{ij}\frac{\lcm(b_i,b_j)}{c_jb_j}g_j+\sum_{i=1}^{t-1}\sum_{j=i+1}^{t}f_{ij}w_{ij}\rho_{ij}\\
=\sum_{j=1}^{t}(h_j-a_jv_j)g_j+\sum_{j=t+1}^{m}h_jg_j+\sum_{i=1}^{t-1}\sum_{j=i+1}^{t}f_{ij}w_{ij}\rho_{ij}.
\end{gather*}
It is clear that $\lmi(h_j-a_jv_j)\LM(f_j)<w_r$ for $1\le j\le t$. On the other hand for a summand $f_{ij}w_{ij}\rho_{ij}$ ($1\le i\le t-1$, $2\le j\le t$) its summands are $f_{ij}w_{ij}q_{ij,l}g_l$. Then we have
\begin{gather*}
\lmi(f_{ij}w_{ij}q_{ij,l})\LM(f_l)\le\LM(w_{ij}S^{'}(f_i,f_j))\\
<\max\{\LM(w_{ij}\frac{\lcm(u_i,u_j)}{u_i}f_i),\LM(w_{ij}\frac{\lcm(u_i,u_j)}{u_j}f_j)\}\\
=\max\{\LM(v_if_i),\LM(v_jf_j)\}\\
\le \max\{\lmi(h_i)\LM(f_i),\lmi(h_j)\LM(f_j)\}\le w_r.
\end{gather*}
This completes the proof.
\end{proof} 

\begin{corollary}\label{syzygy-monomial}
Let $U$ be a submodule of $F$ generated by $\underline{s}$-monomial type terms $f_1,\dots,f_m$. Then $\syz(f_1,\dots,f_m)$ is generated by the relations $r_{ij}=u_{ij}g_i-u_{ji}g_j$ for all $1\le i<j\le m$ for which $f_i$ and $f_j$ involve the same basis element.
\end{corollary}
\begin{proof}
Clearly if $f_i$ and $f_j$ involve the same basis element, then $S^{'}(f_i,f_j)=0$, and so $f_1,\dots,f_m$ form a Gr\"{o}bner basis. Then the claim follows by Theorem~\ref{syzygy-polynomial}.
\end{proof}

\begin{theorem}\label{syzygy-grobner}
Let $U$ be a submodule of $F$ with Gr\"{o}bner basis $G=\{f_1,\dots,f_m\}$ and $f_i$ be $\underline{s}$-monomial type elements. Then the relations $r_{ij}$ arising from the $S^{'}$-elements of the $f_i$ and $f_j$ form a Gr\"{o}bner basis of $\syz(f_1,\dots,f_m)$ with respect to Schreyer monomial order $<_{f_1,\dots,f_m}$. Moreover with this order we have $\LT(r_{ij})=u_{ij}g_i$.
\end{theorem}
\begin{proof}
Since $\lmi(u_{ij})\LM(f_i)=\lmi(u_{ji})\LM(f_j)$ and because $i<j$ we conclude that $\LT(u_{ij}g_i-u_{ji}g_j)=u_{ij}g_i$. On the other hand $\lmi(q_{ij,l})\LM(f_l)\le\LM(S^{'}(f_i,f_j))<\lmi(u_{ij})\LM(f_i)$, then $\LT(q_{ij,l}g_l)<u_{ij}g_i$. Thus it follows that $\LT(r_{ij})=u_{ij}g_i$.

Now we show that the relations $r_{ij}$ form a Gr\"{o}bner basis for $V=\syz(f_1,\dots,f_m)$. Let $g=\sum_{j=1}^{m}r_jg_j$ be an arbitrary relation. Let $\LT(r_jg_j)=a_jv_jg_j$ ($a_j\in R$, $v_j$ is a monomial of $S$) for $j=1,\dots,m$. Then $\LT(g)=a_iv_ig_i$ for some $i$. Now let $g^{'}=\sum_{j}a_jv_jg_j$, where the sum is taken over the set $\mathcal{S}$ of those $j$ for which $\LT(v_jf_j)=\LT(v_if_i)$. Then $j\ge i$ for all $j\in\mathcal{S}$. If we substitute each $g_j$ by $\LT(f_j)$, the sum becomes zero. Therefore $g^{'}$ is a relation of the elements $\LT(f_j)$ with $j\in\mathcal{S}$. Hence by Corollary~\ref{syzygy-monomial}, the element $g^{'}$ is a linear combination of elements of the form $u_{tl}g_t-u_{lt}g_l$ with $t,l\in\mathcal{S}$ and $t<l$. Since $j>i$ for all $j\in\mathcal{S}$ with $j\neq i$, $\LT(g^{'})$ is a linear combination of $u_{ij}g_i$. But each $u_{ij}g_i$ is the leading term of $r_{ij}$. This completes the proof.  
\end{proof}

\begin{theorem}
Let $R$ be a commutative ring and $F=\prod_{i\in\mathcal{A}}R$. Suppose $\underline{s}=s_1,\dots,s_k$ is a permutable weak regular sequence in $R$. Let $U=\langle a_1,\dots,a_n\rangle\subseteq F$, where $a_i$ are $\underline{s}$-monomial type terms. Then $U$ admits a free $R$-resolution 
$$
0\rightarrow F_p\rightarrow F_{p-1}\rightarrow\dots\rightarrow F_1\rightarrow F_0\rightarrow U\rightarrow 0
$$
of length $p\le k$.
\end{theorem}
\begin{proof}
We order every $\underline{s}$-monomial lexicographically when $s_1>s_2>\dots>s_k$. We reorder $a_i$ similar to what is said in~\cite[Corollary 4.17]{ene2011gr}(whenever $\LT(a_i)$ and $\LT(a_j)$ for some $i<j$ involve the same basis element and $\LT(a_i)=ue_k$ and $\LT(a_j)=ve_k$, then $u>v$ with respect to pure lexicographic order). First of all by Buchberger's criterion these terms form a Gr\"{o}bner basis. On the other hand in each step of the resolution, $r_{ij}$ are elements of $\underline{s}$-monomial type and they form a Gr\"{o}bner basis with Schreyer's monomial order (actually here we don't have monomial order on variables and the orders are based on position). We should remark that reordering repeats in any step to remove one factor from the leading terms of $r_{ij}$. By induction on $k-t$ we show that if $a_i$ don't contain factors of $s_1,\dots,s_t$, then $p\le k-t$. If $t=k$, then the leading coefficients of elements of Gr\"{o}bner basis are units. Hence we have a Gr\"{o}bner basis with leading terms of $e_{i_1},\dots,e_{i_l}$, so that the syzygy module is zero. If $t<k$, then we label a Gr\"{o}bner basis as it is mentioned above. This implies that $\syz(a_1,\dots,a_n)$ has a Gr\"{o}bner basis such that $s_1,\dots,s_{t+1}$ do not appear in any of leading terms of the elements of Gr\"{o}bner basis. Thus by induction $\syz(a_1,\dots,a_n)$ has an $R$-resolution of length $\le k-t-1$. This completes the proof.
\end{proof}

\section{Gr\"{o}bner bases and syzygy theorem for a PID}

In this section $R$ is a PID. As we know if $a,b\in R$, then $\lcm(a,b)$ is unique up to unit. For every $a,b\in R$ we choose an arbitrary $\lcm(a,b)$. If $f,g\in F$ involve the same basis element we define
$$
S_R(f,g)=\lcm(\lc(f),\lc(g))S(f,g).
$$
Then $S_R(f,g)$ is unique up to unit and arbitrarily we choose one.

\begin{lemma}\label{pre-main-lemma1-pid}
Let $f_1,\dots,f_m\in F$. Suppose we have a sum $\sum_{i=1}^{m}c_if_i$, where $c_i\in R$ and for all $i$, $\LM(f_i)=x^{\delta}e_l, \delta\in \mathbb{Z}_{\ge 0}^{n}$. If $x^{\delta}e_l>\LT(\sum_{i=1}^{m}c_if_i)$, then $\sum_{i=1}^{m}c_if_i$ is an $R$-linear combination of the $S_R$-elements $S_R(f_i,f_j)$, for $1\le i<j \le m$. We also have $x^{\delta}e_l>\LM\left(S_R(f_i,f_j)\right)$.
\end{lemma}
\begin{proof}
Because the number of $f_i$'s is finite we can choose a finite set of irreducible elements of $R$ such as $q_1,\dots,q_t$, such that the leading coefficients of all $f_i$ are units times a product of $q_i$. On the other hand a finite set of irreducible elements in a PID form a permutable weak regular sequence. If we fix $q_i$ as a permutable weak regular sequence, for each $f_i,f_j$ we have $S^{'}(f_i,f_j)=u_{ij}S_R(f_i,f_j)$, where $u_{ij}$ is a unit. Then the proof is clear.
\end{proof}

\begin{theorem}[Buchberger's criterion]\label{Grobner-basis-pid}
Let $G=\{f_1,...,f_m\}$ be a family of elements in $F$ and $U=\langle f_1,...,f_m\rangle$. Then $G$ is a Gr\"{o}bner basis for $U$ iff for every $i<j$, $S_R(f_i,f_j)\rightarrow_G 0$.
\end{theorem}

Let $G=\{f_1,...,f_m\}$ be a family of elements in $F$ and they form a Gr\"{o}bner basis for a submodule $U$. We have 
$$
S_R(f_i,f_j)= u_{ij}f_i-u_{ji}f_i=q_{ij,1}f_1+\dots+q_{ij,m}f_m,
$$
where $\LM(S_R(f_i,f_j))\ge\lmi(q_{ij,l})\LM(f_l)$, $i<j$, $u_{ij}=\frac{\lcm(b_i,b_j)}{b_i}\frac{\lcm(u_i,u_j)}{u_i}$ ($b_i=\lc(f_i)$ and $u_i=\lm(f_i)$) and the leading terms of such $f_i$ and $f_j$ involve the same basis element.

Let $g_1,\dots,g_m$ be a basis for free module $L$ over $S$. We define
$$
r_{ij}=u_{ij}g_i-u_{ji}g_j-q_{ij,1}g_1-\dots-q_{ij,m}g_m.
$$

\begin{theorem}\label{syzygy-polynomial-pid}
With mentioned notation and condition above $r_{ij}$ ($1\le i<j\le m$), generate syzygies of $f_1,\dots,f_m$. 
\end{theorem}
\begin{proof}
By the same reason explained in the proof of Lemma~\ref{pre-main-lemma1-pid}, $r_{ij}$ arising by $S^{'}$-elements generate syzygies of $U$. But the $r_{ij}$ mentioned above are associate to the  $r_{ij}$ arising from $S^{'}$-elements.
\end{proof}

\begin{corollary}\label{syzygy-monomial-pid}
Let $U$ be a submodule of $F$ generated by terms $f_1,\dots,f_m$. Then $\syz(f_1,\dots,f_m)$ is generated by the relations $r_{ij}=u_{ij}g_i-u_{ji}g_j$ for all $i<j$ for which $f_i$ and $f_j$ involve the same basis element.
\end{corollary}

\begin{theorem}\label{syzygy-grobner-pid}
Let $U$ be a submodule of $F$ with Gr\"{o}bner basis $G=\{f_1,\dots,f_m\}$. Then the relations $r_{ij}$ arising from the $S_R$-elements of the $f_i$ and $f_j$ form a Gr\"{o}bner basis of $\syz(f_1,\dots,f_m)$ with respect to Schreyer monomial order $<_{f_1,\dots,f_m}$. Moreover with this order we have $\LT(r_{ij})=u_{ij}g_i$.
\end{theorem}

\begin{theorem}[Hilbert's syzygy theorem for PID]\label{hilbert-pid}
Let $M$ be a finitely generated $S$-module. Then $M$ admits a free $S$-resolution 
$$
0\to F_p\to F_{p-1}\dots\to F_1\to F_0\to M\to 0
$$
of length $p\le n+1$.
\end{theorem}
\begin{proof}
We know that $M\cong F/U$, where $F$ is a finitely generated free $S$-module. It is enough to prove that $U$ has a free $S$-resolution of length $\le n$. Suppose $e_1,\dots,e_r$ is the basis of $F$ and $f_1,\dots,f_m$ is a Gr\"{o}bner basis of $U$. If $t$ is the largest integer such that the variables $x_1,\dots,x_t$ don't appear in any of the leading terms of $f_i$, then by induction on $n-t$ we prove that $U$ has a free $S$-resolution of length $\le n-t$. If $t=n$, we consider all $f_i$ (if they exist) in which the leading terms involve the basis element $e_1$. If these $f_i$ are $f_{i_1},\dots,f_{i_s}$, then 
$$
\exists r_1,\dots,r_s; r_1\lc(f_{i_1})+\dots r_s\lc(f_{i_s})=\gcd(\lc(f_{i_1}),\dots,\lc(f_{i_s})).
$$
Let $g_1=r_1f_{i_1}+\dots+r_sf_{i_s}$. Now we consider all $f_i$ (if they exist)
whose leading terms involve the basis element $e_2$ and by the same procedure we obtain $g_2$, so that in the maximum case $g_1,\dots,g_r$ form a Gr\"{o}bner basis for $U$ with the same order and its syzygies is zero module.
 
If $t<n$, we assume that the Gr\"{o}bner basis of $f_1,\dots,f_m$ is labeled as described in \cite[Corollary 4.17]{ene2011gr}. Hence by Theorem~\ref{syzygy-grobner-pid} and \cite[Corollary 4.17]{ene2011gr}, $\syz(f_1,\dots,f_m)$ has a Gr\"{o}bner basis such that variables $x_1,\dots,x_{t+1}$ do not appear in any of the leading monomials of the elements of the Gr\"{o}bner basis. Then by induction $\syz(f_1,\dots,f_m)$ has a free $S$-resolution of length $\le n-t-1$. This completes the proof.
\end{proof}

\section{The case of quotient of a PID}
In this section we fix the quotient ring $R/NR$, where $R$ is a PID. We know that we can factor $N=p_1^{n_1}\dots p_k^{n_k}$, where $p_i$ are irreducible elements of $R$. For every $\alpha_i>n_i$, we have 
\begin{gather*}
p_i^{\alpha_i}+NR=p_i^{\alpha_i}+p_1^{n_1}\dots p_k^{n_k}+NR=p_i^{n_i}(p_i^{\alpha_i-n_i}+p_1^{n_1}\dots p_{i-1}^{n_{i-1}}p_{i+1}^{n_{i+1}}\dots p_k^{n_k})+NR.
\end{gather*}
But $\gcd(p_i^{\alpha_i-n_i}+p_1^{n_1}\dots p_{i-1}^{n_{i-1}}p_{i+1}^{n_{i+1}}\dots p_k^{n_k},N)=1$. Hence $p_i^{\alpha_i-n_i}+p_1^{n_1}\dots p_{i-1}^{n_{i-1}}p_{i+1}^{n_{i+1}}\dots p_k^{n_k}+NR$ is a unit. Then all elements of $R/NR$ have a presentation of the 
form $up_1^{\alpha_1}\dots p_k^{\alpha_k}+NR$, where $u+NR$ is a unit and $\alpha_i\le n_i$. On the other hand this presentation is unique because if $up_1^{\alpha_1}\dots p_k^{\alpha_k}+NR=vp_1^{\beta_1}\dots p_k^{\beta_k}+NR$ ($0\le \alpha_i,\beta_i\le n_i$), then $up_1^{\alpha_1}\dots p_k^{\alpha_k}-vp_1^{\beta_1}\dots p_k^{\beta_k}=lp_1^{n_1}\dots p_k^{n_k}$. Without loss of generality we may assume $\alpha_1<\beta_1$. Hence $up_2^{\alpha_2}\dots p_k^{\alpha_k}-vp_1^{\gamma_1}p_2^{\beta_2}\dots p_k^{\beta_k}=lp_1^{\eta_1}p_2^{n_2}\dots p_k^{n_k}$, $\gamma_1,\eta_1>0$. Then $p_1$ divides $up_2^{\alpha_2}\dots p_k^{\alpha_k}$, which is a contradiction.
 
We define $S_{R/NR}$-element of two elements $f,g\in F$ (the base ring is $S=(R/NR)[x_1,\dots,x_n]$). Let the leading terms of $f,g$ involve the same basis element, $\lc(f)=a+NR$ and $\lc(g)=b+NR$, then

$$
S_{R/NR}(f,g)=(\frac{\lcm(a,b)}{a}+NR)\frac{\lcm(\lm(f),\lm(g))}{\lm(f)}f-(\frac{\lcm(a,b)}{b}+NR)\frac{\lcm(\lm(f),\lm(g))}{\lm(g)}g.
$$
We should remark that $S_{R/NR}(f,g)$ is unique up to unit. If $\lc(f)=a+NR$, then we set $\ann(\lc(f))=\frac{N}{a}+NR$. Note that $\ann(\lc(f))$ is unique up to unit. Now we state a similar result to Lemma~\ref{pre-main-lemma1}.

\begin{lemma}\label{pre-main-lemma1-quotient}
Let $f_1,\dots,f_m\in F$. Suppose we have a sum $\sum_{i=1}^{m}c_if_i$, where $c_i\in R/NR$ and for all $i$, $\LM(f_i)=x^{\delta}e_l, \delta\in \mathbb{Z}_{\ge 0}^{n}$. If $x^{\delta}e_l>\LT(\sum_{i=1}^{m}c_if_i)$, then if $m=1$, $c_1f_1=a\ann(\lc(f_1))f_1$ and if $m>1$, then $\sum_{i=1}^{m}c_if_i$ is an $R$-linear combination of the $S_{R/NR}$-elements $S_{R/NR}(f_i,f_j)$, for $1\le i<j \le m$. We also have $x^{\delta}e_l>\LM\left(S_{R/NR}(f_i,f_j)\right)$ and $x^{\delta}e_l>\LM(\ann(\lc(f_i))f_i)$.
\end{lemma}
\begin{proof}
For the case of $m=1$ the claim is clear so we prove the case $m>1$. For each $f_i$ we consider its representative (one of them) in $(R[x_1,\dots,x_n])^r$ and denote it by $g_i$. Now if the leading coefficient of $g_i$ is $u_ip_1^{\alpha_{1,i}}\dots p_k^{\alpha_{k,i}}$, $0\le\alpha_{j,i}\le n_i$, then $\sum_{i=1}^m c_iu_ip_1^{\alpha_{1,i}}\dots p_k^{\alpha_{k,i}}=Nl$. Let $d_i=p_1^{\alpha_{1,i}}\dots p_k^{\alpha_{k,i}}$ and $d_{m+1}=N$. We define $p_i=\frac{g_i}{\lc(g_i)}$ and we have $\sum_{i=1}^{m}c_ig_i=\sum_{i=1}^{m-1}c_iu_id_iS(g_i,g_m)$.

By the same argument as in the proof of Lemma~\ref{pre-main-lemma1},
$$
\sum_{i=1}^{m}c_ig_i=\sum_{i=1}^{m-1}\sum_{j=1}^{m+1}r_{i,j}\lcm(d_i,d_j)S(g_i,g_m).
$$
In the sum above 
$$
r_{i,m}\lcm(d_i,d_m)+r_{i,m+1}\lcm(d_i,d_{m+1})=a\lcm(d_i,d_m), \ a\in R,
$$
so that $\sum_{i=1}^{m}c_if_i$ is a $R$-linear combination of $S_{R/NR}(f_i,f_j)$.
\end{proof}

\begin{theorem}[Buchberger's criterion]\label{Grobner-basis-quotient}
Let $G=\{f_1,...,f_m\}$ be a family of elements in $F$ and $U=\langle f_1,...,f_m\rangle$. Then $G$ is a Gr\"{o}bner basis for $U$ iff for each $i$, $\ann(\lc(f_i))f_i\to_G0$ and for every $i<j$, $S_{R/NR}(f_i,f_j)\rightarrow_G 0$.
\end{theorem}
\begin{proof}
Let $f\in U$, we consider all possible ways that $f=\sum_{i=1}^{m}g_if_i$, where $g_i\in S$, and if $x^{u(i)}e_{u(i)}=\lmi(g_i)\LM(f_i)$ and $x^\delta e_l=\max (x^{u(1)}e_{u(1)},\dots,x^{u(m)}e_{u(m)})$, so $x^\delta e_l$ is minimal. Clearly $x^{\delta}e_l\ge \LT(f)$. If $\LM(f)=x^{\delta}e_l$, then $\LT(f)\in \langle \LT(f_i)\rangle_{i=1}^{m}$, if not, then $x^{\delta}e_l>\LT(f)$. Then we can write
\begin{gather*}
f=\sum_{x^{u(i)}e_{u(i)}=x^\delta e_l}g_if_i + \sum_{x^{u(i)}e_{u(i)}<x^\delta e_l}g_if_i\\
= \sum_{x^{u(i)}e_{u(i)}=x^\delta e_l}\lti(g_i)f_i+\sum_{x^{u(i)}e_{u(i)}=x^\delta e_l}(g_i-\lti(g_i))f_i+\sum_{x^{u(i)}e_{u(i)}<x^\delta e_l}g_if_i.
\end{gather*}
The $x^\delta e_l$ is greater than all monomials appearing in the forth and fifth sums. Then the assumption $x^\delta e_l>\LT(f)$, means that the $x^\delta e_l$ is also greater than leading term of the third sum. On the other hand in the third sum if we have only one summand say $\lti(g_i)f_i$, then $\lti(g_i)f_i=a\ann(\lc(f_i))f_i$ and by the assumption the claim is clear. If we have more than one summand, then the proof is similar to Theorem~\ref{Grobner-basis}.
\end{proof}

Let $G=\{f_1,...,f_m\}$ be a family of elements in $F$ that form a Gr\"{o}bner basis for a submodule $U$. We have 
$$
S_{R/NR}(f_i,f_j)= u_{ij}f_i-u_{ji}f_i=q_{ij,1}f_1+\dots+q_{ij,m}f_m,
$$
where $\LM(S_{R/NR}(f_i,f_j))\ge\lmi(q_{ij,l})\LM(f_l)$, $i<j$, $u_{ij}=(\frac{\lcm(b_i,b_j)}{b_i}+NR)\frac{\lcm(u_i,u_j)}{u_i}$ ($b_i+NR=\lc(f_i)$ and $u_i=\lm(f_i)$) and the 
leading terms of such $f_i$ and $f_j$ involve the same basis element. We also have
$$
\ann(\lc(f_i))f_i=q_{ii,1}f_1+\dots+q_{ii,m}f_m,
$$
where $\LM(\ann(\lc(f_i))f_i)\ge\lmi(q_{ii,l})\LM(f_l)$.

Let $g_1,\dots,g_m$ be a basis for a free module $L$ over $S$. We define
$$
r_{ij}=u_{ij}g_i-u_{ji}g_j-q_{ij,1}g_1-\dots-q_{ij,m}g_m,
$$
and 
$$
r_{ii}=\ann(\lc(f_i))g_i-q_{ii,1}g_1+\dots-q_{ii,m}g_m.
$$

\begin{theorem}\label{syzygy-polynomial-quotient}
With mentioned notation and condition above $r_{ij}$ ($1\le i\le j\le m$) generate $\syz(f_1,\dots,f_m)$. 
\end{theorem}
\begin{proof}
Let $V$ be submodule of $L$ generated by all $r_{ij}$ and let $G=\syz(f_1,\dots,f_m)$. We need to prove that $G\subseteq V$. Let $r=\sum_{j=1}^{m}h_jg_j\in G$.  
	
Let $w_r=\max\{\lmi(h_j)\LM(f_j); j=1,\dots,m\}$. Without loss of generality, we assume that $w_r=\lmi(h_j)\LM(f_j)$ for $j=1,\dots,t$ and $\lmi(h_j)\LM(f_j)<w_r$ for $j=t+1,\dots,m$. Let $b_j+NR=c_jd_j+NR$, where $c_j+NR$ is a unit and $d_j=p_1^{\alpha_{1,j}}\dots p_k^{\alpha_{k,j}}$ ($0\le\alpha_{i,j}\le n_i$). If $\lmi(h_j)=v_j$, and $\lci(h_j)=a_j+NR$, then $w_r=u_jv_je_i$ for $j=1,\dots,t$, ($\LM(f_j)=u_je_i$) and $\sum_{j=1}^{t}a_jc_jd_j+a_{t+1}c_{t+1}d_{t+1}=0$, where $c_{t+1}=1, d_{t+1}=N$. Hence by Corollary~\ref{syzygy}, we have
\begin{gather*}
(a_1c_1,\dots,a_tc_t,a_{t+1}c_{t+1})=\sum_{1\le i<j\le t+1}f_{ij}\frac{\lcm(d_i,d_j)}{d_i}k_i-f_{ij}\frac{\lcm(d_i,d_j)}{d_j}k_j,
\end{gather*}
where $\{k_i|1\le i\le t+1\}$ is the
standard basis for $R^{t+1}$. Then we can write 
\begin{gather*}
(a_1,\dots,a_t,a_{t+1})=\sum_{1\le i<j\le t}\frac{f_{ij}}{\lcm(c_i,c_j)}\frac{\lcm(c_id_i,c_jd_j)}{c_id_i}k_i-\frac{f_{ij}}{\lcm(c_i,c_j)}\frac{\lcm(c_id_i,c_jd_j)}{c_jd_j}k_j\\
+\sum_{i=1}^{t}f_{it+1}\frac{\lcm(d_i,d_{t+1})}{c_id_i}k_i-f_{it+1}\frac{\lcm(d_i,d_{t+1})}{d_{t+1}}k_{t+1}.
\end{gather*}
Let $\frac{f_{ij}}{\lcm(c_i,c_j)}+NR=f_{ij}^{'}+NR$ for $1\le i<j\le t$.
	
For each fixed $1\le i\le t-1$, we have $u_iv_i=u_jv_j$ for $j=i+1,\dots,t$ so that there exist monomials $w_{ij}$ such that $v_j=w_{ij}\frac{\lcm(u_i,u_j)}{u_j}$ for $j=i+1,\dots,t$. 
	
Assume for the moment that we already know that $w_{r^{'}}<w_r$ for the relation
$$
r^{'}=r-\sum_{i=1}^{t-1}\sum_{j=i+1}^{t}(f_{ij}^{'}+NR)w_{ij}r_{ij}-\sum_{i=1}^{t}(f_{it+1}+NR)v_ir_{ii}.
$$
On the other hand we see that $\ann(\lc(f_i))=\frac{\lcm(d_i,d_{t+1})}{c_id_i}+NR$. Then using induction the claim follows similar to Theorem~\ref{syzygy-polynomial}.
\end{proof} 
\begin{corollary}\label{syzygy-monomial-quotient}
Let $U$ be a submodule of $F$ generated by terms $f_1,\dots,f_m$. Then $\syz(f_1,\dots,f_m)$ is generated by the relations $\ann(\lc(f_i))g_i$ and $r_{ij}=u_{ij}g_i-u_{ji}g_j$ for all $1\le i<j\le m$ for
 which $f_i$ and $f_j$ involve the same basis element.
\end{corollary}

\begin{theorem}\label{syzygy-grobner-quotient}
Let $U$ be a submodule of $F$ with Gr\"{o}bner basis $G=\{f_1,\dots,f_m\}$. Then the relations $r_{ij}$ for $1\le i\le j\le m$ form a Gr\"{o}bner basis of $\syz(f_1,\dots,f_m)$ with respect to Schreyer monomial order $<_{f_1,\dots,f_m}$. Moreover with this order we have $\LT(r_{ij})=u_{ij}g_i$ when $i<j$ and $\LT(r_{ii})=\ann(\lc(f_i))g_i$.
\end{theorem}

\begin{theorem}[Hilbert's syzygy theorem for quotient of PID]\label{hilbert-quotient}
Let $M$ be a finitely generated $S$-module. Then $M$ admits a free $S$-resolution
$$
\dots\to F_p\to F_p\to F_p\to F_{p-1}\dots\to F_1\to F_0\to M\to 0
$$
where $p\le n+1$.
\end{theorem}
\begin{proof}
The proof is similar to Theorem~\ref{hilbert-pid}. The only step at which the argument is different is when the leading term of Gr\"{o}bner basis does not contain any variables. Since in $R/NR$ we have B\'ezout identity, then $g_i$ form a Gr\"{o}bner basis for $U$, where $g_i$ are the same as in the proof of the mentioned theorem. If the leading coefficients of all $g_i$ are units, then the syzygies are zero modules. It is enough to map a free module onto $U$ such that the kernel of this map is zero. In the next step we map the same free module to zero. The next step we map the same free module to itself and so on. Otherwise we may assume that the leading term of $g_i$ for $1\le i\le s$ is not a unit and the leading term of $g_i$ for $i\ge s+1$ is a unit. Hence we map $\oplus S$ onto $U$ by sending $h_i$ to $g_i$ ($\{h_i\}$ is standard basis of the free module). Then by Theorem~\ref{syzygy-polynomial-quotient} the syzygies are generated by $l_1,\dots,l_s$, where the leading term of $l_i$ is $\ann(\lc(g_i))h_i$. It is enough to map $\oplus_{i=1}^s S$ by sending the standard basis to $l_i$, and map $\oplus_{i\ge s+1}S$ to zero. We can continue this procedure.
\end{proof}

\begin{corollary}
Let $M$ be a finitely generated $S$-module. If $M\cong F/U$ and the leading coefficients of a Gr\"{o}bner basis of $U$ with a monomial order all are nonzero divisors, then $M$ admits a free $S$-resolution
$$
0\to F_p\to F_{p-1}\dots\to F_1\to F_0\to M\to 0
$$
of length $p\le n+1$.
\end{corollary}

\section{Gr\"{o}bner bases and syzygy theorem for direct product of principal ideal rings}
In this section we fix $R=\prod_{i=1}^{p}R_i/N_iR_i$, where $R_i$ are PID and $N_i\in R_i$, moreover $N_i$ is not a unit but it is not necessarily nonzero. Then $R$ is a direct sum of PIDs and quotients of PIDs. As before $S=R[x_1,\dots,x_n]$, $F=\oplus_{i=1}^rS$ with standard basis $e_1,\dots,e_r$. If $U$ is a submodule of $F$, then a Gr\"{o}bner basis of $U$ exists. By $(a_i+N_iR_i)$ we mean an element of $R$, whose $i$th component is $a_i+N_iR_i$.

Let $f,g\in F$, $\lc(f)=(a_i+N_iR_i)$ and $\lc(g)=(b_i+N_iR_i)$. Then we define
$$
S_R(f,g)=(\frac{\lcm(a_i,b_i)}{a_i}+N_iR_i)\frac{\lcm(\lm(f),\lm(g))}{\lm(f)}f-(\frac{\lcm(a_i,b_i)}{b_i}+N_iR_i)\frac{\lcm(\lm(f),\lm(g))}{\lm(g)}g.
$$
We should remark that if $a_i+N_iR_i$ or $b_i+N_iR_i$ is zero, then we define $\frac{\lcm(a_i,b_i)}{a_i}+N_iR_i=\frac{\lcm(a_i,b_i)}{b_i}+N_iR_i=N_iR_i$. We also remark that $S_R(f,g)$ is unique up to unit. Let $f\in F$. By $(d)_j$ we mean an element in $R$ whose $j$th component is $d$ and whose other components are zero.

\begin{lemma}\label{pre-main-lemma1-general}
Let $f_1,\dots,f_m\in F$. Suppose we have a sum $\sum_{i=1}^{m}c_if_i$, where $c_i\in R$ and for all $i$, $\LM(f_i)=x^{\delta}e_l, \delta\in \mathbb{Z}_{\ge 0}^{n}$. If $x^{\delta}e_l>\LT(\sum_{i=1}^{m}c_if_i)$, then $\sum_{i=1}^{m}c_if_i$ is an $R$-linear combination of the $\ann(\lc(f_i))f_i$ and $S_R$-elements $S_R(f_i,f_j)$, for $1\le i<j \le m$. We also have $x^{\delta}e_l>\LM\left(S_R(f_i,f_j)\right)$ and $x^{\delta}e_l>\LM\left(\ann(\lc(f_i))f_i\right)$.
\end{lemma}
\begin{proof}
We assume that $j$th component of $c_i$ is $c_{ij}$. Then $\sum_{i=1}^{m}c_if_i=\sum_{i=1}^{m}\sum_{j=1}^{p}(c_{ij})_jf_i$. Hence for each $j$ we consider the subsum $\sum_{i=1}^{m}(c_{ij})_jf_i$. We see that $x^\delta e_l>\LT(\sum(c_{ij})_jf_i)$. Let $\lc(f_i)_j$ be the $j$th component of $\lc(f_i)$. Now if all $(c_{ij})_j$ are zero, then we don't consider this subsum. If $\lc(f_i)_j$ is zero, then $(c_{ij})_j=(c_{ij})_j(1)_j$, so that $(c_{ij})_jf_i=(c_{ij})_j\ann(\lc(f_i))f_i$. Then we can consider $\beta$'s such that $\lc(f_\beta)_j\neq 0$. We have $x^\delta e_l>\LT(\sum_\beta (c_{\beta j})f_\beta)$. Now if there is only one $\beta$, then $(c_{\beta j})f_\beta=(a)_j\ann(\lc(f_\beta))f_\beta$. If there is more than one $\beta$, then $\sum c_{\beta j}\lc(f_\beta)_j=0$. Let
\begin{gather*}
S_j(f,g)=(\frac{\lcm(\lc(f)_j,\lc(g)_j)}{\lc(f)_j}+N_jR_j)_j\frac{\lcm(\lm(f),\lm(g))}{\lm(f)}f\\
-(\frac{\lcm(\lc(f)_j,\lc(g)_j)}{\lc(g)_j}+N_jR_j)_j\frac{\lcm(\lm(f),\lm(g))}{\lm(g)}g.
\end{gather*}
Thus by the same argument as Lemma~\ref{pre-main-lemma1-pid} and Lemma~\ref{pre-main-lemma1-quotient}, we have
$$
\sum (c_{\beta j})_jf_\beta=\sum(a_\gamma)_jS_j(f_\eta,f_\theta).
$$
On the other hand 
$$
(a_\gamma)_jS_j(f_\eta,f_\theta)=(a_\gamma)_jS_R(f_\eta,f_\theta).
$$
This completes the proof.
\end{proof}

Let $G=\{f_1,...,f_m\}$ be a family of elements in $F$ that form a Gr\"{o}bner basis for a submodule $U$. We have 
$$
S_R(f_\alpha,f_\beta)= u_{\alpha\beta}f_\alpha-u_{\beta\alpha}f_\beta=q_{\alpha\beta,1}f_1+\dots+q_{\alpha\beta,m}f_m,
$$
where $\LM(S_R(f_\alpha,f_\beta))\ge\lmi(q_{\alpha\beta,l})\LM(f_l)$, $\alpha<\beta$, $u_{\alpha\beta}=(\frac{\lcm(a_i,b_i)}{a_i}+N_iR_i)\frac{\lcm(u_\alpha,u_\beta)}{u_\alpha}$ ($(a_i+N_iR_i)=\lc(f_\alpha)$ and $u_\alpha=\lm(f_\alpha)$) and the leading terms of such $f_\alpha$ and $f_\beta$ involve the same basis element. We also have
$$
\ann(\lc(f_\alpha))f_\alpha=q_{\alpha\alpha,1}f_1+\dots+q_{\alpha\alpha,m}f_m,
$$
where $\LM(\ann(\lc(f_\alpha)))f_\alpha\ge\lmi(q_{\alpha\alpha,l})\LM(f_l)$.

Let $g_1,\dots,g_m$ be a basis for free module $L$ over $S$. We define
$$
r_{\alpha\beta}=u_{\alpha\beta}g_\alpha-u_{\beta\alpha}g_\beta-q_{\alpha\beta,1}g_1-\dots-q_{\alpha\beta,m}g_m,
$$
and 
$$
r_{\alpha\alpha}=\ann(\lc(f_\alpha))g_\alpha-q_{\alpha\alpha,1}g_1+\dots-q_{\alpha\alpha,m}g_m.
$$

\begin{theorem}\label{syzygy-polynomial-general}
With the mentioned notation and condition above $r_{\alpha\beta}$ ($1\le\alpha\le\beta\le m$) generate $\syz(f_1,\dots,f_m)$.
\end{theorem}
\begin{proof}
Let $V$ be the 
submodule of $L$ generated by all $r_{\alpha\beta}$ and let $G=\syz(f_1,\dots,f_m)$. We need to prove that $G\subseteq V$. Let $r=\sum_{\alpha=1}^{m}h_\alpha g_\alpha\in G$.

Let $w_r=\max\{\lmi(h_\alpha)\LM(f_\alpha); \alpha=1,\dots,m\}$. Without loss of generality, we assume that $w_r=\lmi(h_\alpha)\LM(f_\alpha)$ for $\alpha=1,\dots,t$ and $\lmi(h_\alpha)\LM(f_\alpha)<w_r$ for $\alpha=t+1,\dots,m$. If $\lmi(h_\alpha)=v_\alpha$, then $w_r=u_\alpha v_\alpha e_\beta$ for $\alpha=1,\dots,t$, ($\LM(f_\alpha)=u_\alpha e_\beta$). We assume that $\gamma$-component of $\lc(f_\alpha)$ (resp. $\lci(h_\alpha)$) is $\lc(f_\alpha)_\gamma$ (resp. $\lci(h_\alpha)_\gamma$). The rest of the proof is similar to the proof of Theorem~\ref{syzygy-polynomial} and Theorem~\ref{syzygy-polynomial-quotient}. We consider 
$$
r^{'}=r-\sum_{i=1}^{t-1}\sum_{j=i+1}^{t}(c_{ij\gamma})w_{ij}r_{ij}-\sum_{i=1}^{t}(d_{ii\gamma})v_ir_{ii}.
$$
In the following we will say what are values of $c_{ij\gamma}$ and $d_{ii\gamma}$. For an arbitrary $\gamma$ we have two cases:

(i) $N_\gamma=0$. Therefore, in this component we deal with a PID. If for some $1\le\alpha\le t$ we have $\lci(h_\alpha)_\gamma=0$, then we have $c_{i\alpha\gamma}=c_{\alpha j\gamma}=d_{\alpha\alpha\gamma}=0$. If $\lci(h_\alpha)_\gamma\neq 0$ and $\lc(f_\alpha)_\gamma=0$, then $c_{i\alpha\gamma}=c_{\alpha j\gamma}=0$ and $d_{\alpha\alpha\gamma}=\lci(h_\alpha)_\gamma$. For remaining $1\le i,j\le t$, we have $d_{ii\gamma}=0$ and $c_{ij\gamma}$ can be obtained by using Corollary~\ref{syzygy}.

(ii) $N_\gamma\neq 0$. Therefore, in this component we deal with a quotient of a PID. Again the argument is similar to the case (i) when $\lci(h_\alpha)_\gamma=0$ or $\lci(h_\alpha)_\gamma\neq 0$ and $\lc(f_\alpha)_\gamma=0$. For the remaining $1\le i,j\le t$, similar to the proof of Theorem~\ref{syzygy-polynomial-quotient}, we have $c_{ij\gamma}=f_{ij}^{'}+N_\gamma R_\gamma$ and $d_{ii\gamma}=f_{it+1}+N_\gamma R_\gamma$.  
\end{proof}

\begin{corollary}\label{syzygy-monomial-general}
Let $U$ be a submodule of $F$ generated by terms $f_1,\dots,f_m$. Then $\syz(f_1,\dots,f_m)$ is generated by the relations $\ann(\lc(f_\alpha))g_\alpha$ and $r_{\alpha\beta}=u_{\alpha\beta}g_\alpha-u_{\beta\alpha}g_\beta$ for all $1\le\alpha<\beta\le m$ for which $f_\alpha$ and $f_\beta$ involve the same basis element.
\end{corollary}

\begin{theorem}\label{syzygy-grobner-general}
Let $U$ be a submodule of $F$ with Gr\"{o}bner basis $G=\{f_1,\dots,f_m\}$. Then the relations $r_{\alpha\beta}$ for $1\le\alpha\le\beta\le m$ form a Gr\"{o}bner basis of $\syz(f_1,\dots,f_m)$ with respect to Schreyer monomial order $<_{f_1,\dots,f_m}$. Moreover with this order we have $\LT(r_{\alpha\beta})=u_{\alpha\beta}g_\alpha$ when $\alpha<\beta$ and $\LT(r_{\alpha\alpha})=\ann(\lc(f_\alpha))g_\alpha$.
\end{theorem}

\begin{theorem}[Hilbert's syzygy theorem for finite direct product]\label{hilbert-general}
Let $M$ be a finitely generated $S$-module. Then $M$ admits a free $S$-resolution
$$
\dots\to F_p\to F_p\to F_p\to F_{p-1}\dots\to F_1\to F_0\to M\to 0
$$
where $p\le n+1$.
\end{theorem}
\begin{proof}
Regarding the proof of Theorem~\ref{hilbert-pid}, it is enough to do the last step when there are no variables in the leading term of $f_1,\dots,f_m$ and they form a Gr\"{o}bner basis for $U$. In each component of elements in $R$ we have B\'ezout identity. Then we have B\'ezout identity on elements of $R$. Therefore similarly to the proof of Theorem~\ref{hilbert-pid}, we obtain $g_i$ which form a Gr\"{o}bner basis for $U$. If all $g_i$ are nonzero divisors then we are done. Otherwise, the argument is similar to Theorem~\ref{hilbert-quotient}.
\end{proof}

\begin{corollary}[Hilbert's syzygy theorem for finite direct product]
Let $M$ be a finitely generated $S$-module. If $M\cong F/U$ and the leading coefficients of a Gr\"{o}bner basis of $U$ with a monomial order all are nonzero divisors then $M$ admits a free $S$-resolution
$$
0\to F_p\to F_{p-1}\dots\to F_1\to F_0\to M\to 0
$$
of length $p\le n+1$.
\end{corollary}

In the rest of this section $R=\prod_{i\in\mathcal{A}}R_i/N_iR_i$, where $\mathcal{A}$ is an infinite set and $R_i$ are PIDs. Before we state the main result we make the following observation.

\begin{remark}
If $U$ is a submodule of $F$ and $\LT(U)$ is finitely generated then we have the same $r_{\alpha\beta}$. We have a similar result to Theorem~\ref{syzygy-polynomial-general} with a similar proof. We should also observe that Buchberger's criterion is valid for infinite direct products. In the proof of this theorem we just need to explain that since the number of $S_R$-elements is finite, at the end we can obtain a finite sum of these elements. 
  
\end{remark}

Then we have the following results.

\begin{theorem}[Hilbert's syzygy theorem for direct product]\label{hilbert-direct-product}
Let $M$ be a finitely generated $S$-module. Suppose $M\cong F/U$ and $\LT(U)$ with a monomial order is finitely generated. Then $M$ admits a free $S$-resolution
$$
\dots\to F_p\to F_p\to F_p\to F_{p-1}\dots\to F_1\to F_0\to M\to 0
$$
where $p\le n+1$.
\end{theorem}

\begin{corollary}[Hilbert's syzygy theorem for direct product]
Let $M$ be a finitely generated $S$-module. If $M\cong F/U$, $LT(U)$ is finitely generated and the leading coefficients of a Gr\"{o}bner basis all are nonzero divisors, then $M$ admits a free $S$-resolution
$$
0\to F_p\to F_{p-1}\dots\to F_1\to F_0\to M\to 0
$$
of length $p\le n+1$.
\end{corollary}

\section{Gr\"{o}bner bases for solvable rings}

\begin{definition}
A ring $R$ is solvable when given $a,a_1,\dots,a_m\in R$, we can determine whether $a\in\langle a_1,\dots,a_m\rangle$ and if it is, to find $b_1,\dots,b_m\in R$ such that $a=a_1b_1+\dots+a_mb_m$. 
\end{definition}

Note that every Euclidean domain is solvable.

\begin{remark}
If $R$ is solvable, then $R/NR$ is solvable. Because the question $a+NR\in\langle a_1+NR,\dots,a_k+NR\rangle$ and finding solutions goes to the question $a\in\langle a_1,\dots,a_k,N\rangle$ and finding solutions.  
\end{remark}

Now we state the division algorithm on solvable rings.

\begin{lemma}\label{division-algorithm}
Let $R$ be a solvable ring and $E=(f_1,\dots,f_m)$ be an ordered $m$-tuple of elements of $F$, and let $f\in F$. Consider the following algorithm:
	
1. Order $m$-tuples $(i_1,\dots,i_m)$ ($i_j\in\{0,1\}$) lexicographically.
	
2. $g_1=g_2=\dots=g_m=0$, $r=0$, and $g=f$.
	
3. If $\LT(g)$ is generated by $\LT(f_i)$, then find smallest $(i_1,\dots,i_m)$ ($i_j\in\{0,1\}$) such that $\LT(g)=\sum_{j=1}^{m}r_jh_ji_j\LT(f_j)$ ($h_j$ are monomials). Replace $g_j$ by $g_j+r_jh_ji_j$ and $g$ by $g-\sum_{j=1}^{m}r_jh_ji_jf_j$. 
	
4. Repeat step 3 until $\LT(g)$ is not generated by $\LT(f_j)$. Then replace $r$ by $r+\LT(g)$ and $g$ by $g-\LT(g)$.
	
5. If now $g\neq 0$, start again with step 3. If $g=0$ stop the algorithm.
	
This is an algorithm which gives us an $m$-tuple $(g_1,\dots,g_m)\in S^m$ and an $r\in F$ with
$$
f=g_1f_1+\dots+g_mf_m+r,
$$
and such that the following conditions are satisfied.
	
a. If $r\neq 0$, then none of the terms of $r$ is generated by $\LT(f_1),\dots,\LT(f_m)$.
	
b. We have $\LM(f)\ge\lmi(g_i)\LM(f_i)$. 
	
We call $r$ a remainder of $f$ on division by $E$ and denote it by $\overline{f}^E$.
\end{lemma}

The algorithm terminates in finitely many steps, because in each step the leading term becomes strictly smaller.

In the rest of this section we fix $R=\prod_{i=1}^{p}R_i/N_iR_i$, where $R_i$ are PID and $N_i\in R_i$, moreover $N_i$ is not a unit but it is not necessarily nonzero. Then $R$ is a direct product of PIDs and quotients of PIDs. We also assume that every $R_i$ is solvable.

Now we can state Buchberger's algorithm.

\begin{theorem}\label{buchberger-algorithm}
Let $f_1,\dots,f_m\in F$. Consider the following algorithm:
	
1. Initially consider an ordered $m$-tuple $H=(f_1,\dots,f_m)$.
	
2. For each pair $\{p,q\}$, $p\neq q$ in $H$ do first $\mathcal{S}=\overline{S_R(p,q)}^H$ and if $\mathcal{S}\neq 0$, then add $\mathcal{S}$ to $H$. Second compute $\mathcal{S}=\overline{\ann(\lc(f))f}^H$ ($f=p,q$) and in each step if $\mathcal{S}\neq 0$, then add $\mathcal{S}$ to $H$.

3. Repeat 2 until all remainders are zero.
	
This is an algorithm that gives us a Gr\"{o}bner basis for the submodule $U=\langle f_1,\dots,f_m\rangle$.
\end{theorem}

The algorithm terminates in finitely many steps since the calculation is in a finite direct sum which is a Noetherian ring.

\begin{remark}
By \cite[Lemma 10, Corollary 11]{hungerford1968structure}, every principal ideal ring is a finite direct sum of quotients of PIDs. Then all arguments about finite direct product of quotients of PIDs are valid for a principal ideal ring. Also, if a principal ideal ring is solvable we can state Buchberger's algorithm. 
\end{remark}

\begin{example}
Let $S=\mathbb{Z}/2\mathbb{Z}\times\mathbb{Z}/4\mathbb{Z}\times\mathbb{Z}/8\mathbb{Z}[X,Y]$ and we have lexicographic order. Let 
$$
I=\langle f_1=(0,2,2)X^2+(1,1,0),f_2=(1,2,4)Y+(0,3,0),f_3=(1,0,0)\rangle.
$$
For brevity by $\ann(f_j)$ we mean $\ann(\lc(f_j))f_j$.

1. $H=(f_1,f_2,f_3)$, $f_4=\overline{S_R(f_1,f_2)}^H=(0,1,0)X^2+(0,1,0)XY$, $H=(f_1,f_2,f_3,f_4)$, $f_5=\overline{\ann(f_2)}^H=(0,2,0)$, $H=(f_1,f_2,f_3,f_4,f_5)$, $\overline{\ann(f_1)}^H=0$.

2. $S_R(f_1,f_3)=S_R(f_2,f_3)=\ann(f_3)=0$.

3. $f_6=\overline{S_R(f_1,f_4)}^H=(0,1,0)X$, $H=(f_1,f_2,f_3,f_4,f_5,f_6)$, $\ann(f_4)=0$.

4. $\overline{S_R(f_2,f_4)}^H=0$, $S_R(f_3,f_4)=0$.

5. $\overline{S_R(f_1,f_5)}^H=\overline{S_R(f_4,f_5)}^H=0$, $f_7=\overline{S_R(f_2,f_5)}^H=(0,-1,0)$, $H=(f_1,f_2,f_3,f_4,f_5,f_6,f_7)$, $\ann(f_5)=0$.

Finally we see that all other remainders are zero. We can also remove some of these polynomials and remaining polynomials which are $f_1,f_2,f_3,f_7$ form a Gr\"{o}bner basis.
\end{example}

\begin{example}
Let $S=\mathbb{Z}/2\mathbb{Z}\times\mathbb{Z}/4\mathbb{Z}\times\mathbb{Z}/8\mathbb{Z}[X,Y]$ and $F=S\oplus S$. Suppose $F$ has lexicographic order given priority to the position. Let
$$
U=\langle f_1=(0,2,1)XY^2e_2+(0,1,0)e_2, f_2=(1,2,2)X^2Ye_1+(0,1,4)Xe_2, f_3=(1,0,1)e_2\rangle.
$$
For brevity by $\ann(f_j)$ we mean $\ann(\lc(f_j))f_j$.

1. $H=(f_1,f_2,f_3)$, $S_R(f_1,f_3)=0$, $f_4=\ann(f_1)=(0,2,0)e_2$, $H=(f_1,f_2,f_3,f_4)$.

2. $\overline{\ann(f_2)}^H=0$.

3. $f_5=S_R(f_1,f_4)=(0,1,0)e_2$. $H=(f_1,f_2,f_3,f_4,f_5)$.

Finally we see that all other remainders are zero. We can also see that $f_2,f_3,f_5$ is a Gr\"{o}bner basis.
\end{example}

\begin{example}
Let $R=\mathbb{Z}\times\mathbb{Z}$. Consider the ideal $I=\langle (2,0)x^2y+(1,2),(0,3)xy^2+(1,1)y,(3,4)x\rangle\subset S=R[x,y]$. After calculation we obtain the Gr\"{o}bner basis below for $I$:
$$
f_1=(0,3)xy^2+(1,1)y, f_2=(0,2), f_3=(1,0).
$$
We have $r_{13}=r_{23}=0$, $r_{12}=(0,2)g_1-(0,3)xy^2g_2-yg_2$, $r_{11}=(1,0)g_1-yg_3$, $r_{22}=(1,0)g_2$ and $r_{33}=(0,1)g_3$. Using B\'ezout identity we obtain the Gr\"{o}bner basis below for syzygies:
$$
(1,2)g_1-(0,3)xy^2g_2-yg_2-yg_3, (1,0)g_2, (0,1)g_3.
$$
Then we obtain the following $S$-resolution for $I$:
$$
\dots\to S\oplus S\to S\oplus S\to S\oplus S\oplus S\to S\oplus S\oplus S\to I.
$$
\end{example}
We remark that dimension of $\mathbb{Z}\times\mathbb{Z}$ is 1.

\end{document}